\documentclass[english]{article}
\usepackage[T1]{fontenc}
\usepackage[utf8]{inputenc}
\usepackage{geometry}
\usepackage{amsmath}
\usepackage{amsthm}
\usepackage{amssymb}
\usepackage{stackrel}
\usepackage{graphicx}
\graphicspath{ {images/} }

\usepackage{setspace}
\makeatletter

\theoremstyle{plain}
\newtheorem{thm}{\protect\theoremname}[section]
  \theoremstyle{definition}
  \newtheorem{defn}[thm]{\protect\definitionname}
  \theoremstyle{remark}
  \newtheorem{rem}[thm]{\protect\remarkname}
  \theoremstyle{plain}
  \newtheorem{lem}[thm]{\protect\lemmaname}
  \theoremstyle{plain}
  \newtheorem{prop}[thm]{\protect\propositionname}
  \theoremstyle{remark}
  \newtheorem*{rem*}{\protect\remarkname}

\makeatother

\usepackage{babel}
  \providecommand{\definitionname}{Definition}
  \providecommand{\lemmaname}{Lemma}
  \providecommand{\propositionname}{Proposition}
  \providecommand{\remarkname}{Remark}
\providecommand{\theoremname}{Theorem}

\begin{document}
\global\long\def\pp{\varphi}
 \global\long\def\sk{\underrightarrow{Ker}}
 \global\long\def\abs#1{\vert#1\vert}
 \global\long\def\agm#1{\underset{#1}{*}}
 \global\long\def\FF{\mathbb{F}}
 \global\long\def\aij#1#2#3{#1_{#3}^{(#2)}}
\global\long\def\cayfl{X(\FF_{\ell})}
\global\long\def\nn{n\in\mathbb{\mathbb{N}}}

\title{Canonicality of Makanin-Razborov Diagrams - Counterexample}

\author{Gili Berk}
\maketitle
\begin{abstract}
Sets of solutions to finite systems of equations in a free group,
are equivalent to sets of homomorphisms from a fixed f.p. group into
a free group. The latter can be encoded in a diagram, the construction
of which is valid also for f.g. groups. The diagram is known to be
canonical for a fixed f.g. group with a fixed generating set. In this
paper we prove that the construction depends on the chosen generating
set of the given f.g. group.
\end{abstract}

\section*{Introduction}

For a given finite system of equations $\Phi$ over a free group $\FF_{k}$,
there is a natural associated f.g. group $G(\Phi)$. If the system
$\Phi$ is defined by the coefficients $a_{1},...,a_{k}$, the unknowns
$x_{1},...,x_{n}$ and the equations $\{w_{i}(a_{1},...,a_{k},x_{1},...,x_{n})=1\}_{i=1}^{s}$,
then $G(\Phi)=\langle a_{1},...,a_{k},x_{1},...,x_{n}\vert\{w_{i}\}_{i=1}^{s}\rangle$.
If there are no coefficients then $G(\Phi)=\langle x_{1},...,x_{n}\vert\{w_{i}\}_{i=1}^{s}\rangle$.
There is a correspondence between solutions of the system $\Phi$
and homomorphisms $h:G(\Phi)\rightarrow\FF_{k}$ (for which the restriction
$h(a_{j})=a_{j}$ holds $\forall1\leq j\leq k$, in the case with
coefficients). Therefore, the study of one is equivalent to the study
of the other. In addition, every f.g. group $G$ (not necessarily
f.p.) has a canonical finite (restricted) factor set $\{q_{i}:G\rightarrow L_{i}\}_{i=1}^{m}$
through which any (restricted) homomorphism $h:G\rightarrow\FF_{k}$
factors, where the $L_{i}$ are (restricted) limit groups (see Section
7 in \cite{key-7}). So in order to understand the set of (restricted)
homomorphisms from a f.g. group into the a free group, it is sufficient
to study the set of (restricted) homomorphisms from a (restricted)
limit groups into a free group. 

For a given (restricted) limit group $G$ with a finite generating
set $\boldsymbol{g}$, the set of (restricted) homomorphisms from
$G$ to $\FF_{k}$ is encoded in the canonical (restricted) Makanin-Razborov
diagram (see \cite{key-7}). The diagram is constructed iteratively,
so that each level is comprised of (restricted) maximal shortening
quotients of freely indecomposable components of groups from the previous
level. These maximal shortening quotients are taken with respect to
generating sets inherited from $\boldsymbol{g}$. Therefore, to conclude
whether or not a (restricted) Makanin-Razborov diagram is dependent
on the generating set, it is sufficient to examine the canonicality
of the set of (restricted) maximal shortening quotients of a freely
indecomposable (restricted) limit group (up to isomorphism of shortening
quotients).

It transpires that in the restricted case a counterexample exists.
We give a restricted limit group $D_{w,z}$ and its essential JSJ
decomposition, and compute its strict restricted shortening quotients
with respect to two generating sets $\boldsymbol{g}$ and $\boldsymbol{u}$.
The generating sets $\boldsymbol{g}$ and $\boldsymbol{u}$ are chosen
such that there cannot be an isomorphism of shortening quotients between
a strict restricted $\boldsymbol{g}$-shortening quotient and a strict
restricted $\boldsymbol{u}$-shortening quotient, because they have
different abelianizations.

The lack of canonicality of MR-diagrams has the additional implication
that maximal shortening quotients of a limit group in one generating
set are not always isomorphic (as shortening quotients) to shortening
quotients in another generating set. Otherwise, the canonicality of
the MR-diagram - or rather of a modified diagram in which only properly
maximal shortening quotients appear - would easily ensue (Lemma \ref{lem:TFAE, section 4}).
However, for some limit groups it can be shown that certain maximal
shortening quotients in one generating set are indeed isomorphic to
shortening quotients in any other generating set.

The paper is organised as follows: Section \ref{sec:Preliminaries}
provides some terminology, notation and facts. Section \ref{sec:Counterexample}
is devoted to the description of the counterexample. That section
uses a particular word $w$ studied by S. V. Ivanov \cite{key-3}
and results of S. Heil \cite{key-2} regarding  JSJ forms of doubles
of free groups of rank 2. Section \ref{sec:examples} describes some
limit groups for which certain maximal shortening quotients in one
generating set are generator-independent shortening quotients. Test
sequences play a significant role in the arguments presented in that
section.

I would like to thank Zlil Sela and Chloé Perin for their invaluable
help and insight.

\section{Preliminaries\label{sec:Preliminaries}}

Throughout this paper, $\mathbb{F}$ will denote some finitely generated
free group with a fixed free generating set, and $X(\FF)$ will denote
its Cayley graph with respect to this generating set. For a given
pointed simplicial $G$-tree $(T,t)$ and an element $g\in G$, denote
$\vert g\vert_{T}=d(t,g.t)$ the \emph{displacement length}, where
$d$ is the simplicial metric on $T$. If $G=\mathbb{F}$ (with the
fixed set of generators) then the length of $g\in\mathbb{F}$ can
be measured in the corresponding Cayley graph $X(\mathbb{F})$ with
respect to the basepoint $1_{\FF}$. For simplicity, when the length
is measured in $X(\FF)$ we will write $\abs g$ rather than $\abs g_{X(\FF)}$.
The \emph{translation length} of an element $g\in G$ in the $G$-tree
T is $tr_{T}(g)=\underset{h\in G}{min}\,\abs{hgh^{-1}}_{T}$.

We refer to JSJ decompositions in the sense of Rips and Sela \cite{key-1}.

\subsubsection*{Shortness with respect to a generating set}

It will be helpful to use terminology and notation which keep track
of the generating set with respect to which the shortening is done.
\begin{defn}
\cite{key-7} For a group $G$ with a finite generating set $\boldsymbol{g}=(g_{1},...,g_{k})$:
\end{defn}
\begin{itemize}
\item A homomorphism $h\in Hom(G,\mathbb{F})$ is called \emph{$\boldsymbol{g}$-shortest}
if $\underset{i}{max}\vert h(g_{i})\vert\leq\underset{i}{max}\vert\iota_{c}\circ h\circ\varphi(g_{i})\vert$
for all \linebreak{}
$c\in\mathbb{F}$ and all $\varphi\in Mod(L)$, where $\iota_{c}$
is the conjugation by $c$.
\item Suppose $G$ is freely indecomposable, and let $L$ be a quotient
of $G$ with the quotient map\linebreak{}
$\eta:G\rightarrow L$. The pair $(L,\eta)$
is called a \emph{$\boldsymbol{g}$-shortening quotient} of $G$ if
$Ker\eta=\underrightarrow{Ker}h_{n}$ for some stable sequence of
$\boldsymbol{g}$-shortest homomorphisms $\{h_{n}\}_{n\in\mathbb{N}}\subseteq Hom(G,\mathbb{F})$.
Denote the set of \emph{$\boldsymbol{g}$}-shortening quotients of
$G$ as $SQ(G,\boldsymbol{g})$.
\item The set of $\boldsymbol{g}$-shortening quotients of $G$ is partially
ordered by the relation $\leq_{\boldsymbol{g}}$ given by\linebreak{}
 $(L,\eta)\leq_{\boldsymbol{g}}(M,\pi)$ if there exists an epimorphism
$\sigma:M\twoheadrightarrow L$ such that $\eta=\sigma\circ\pi$.
\emph{Maximal $\boldsymbol{g}$-shortening quotient} are maximal elements
in $SQ(G,\boldsymbol{g}$) with regard to the partial order $\leq_{\boldsymbol{g}}$.
Denote the set of maximal $\boldsymbol{g}$-shortening quotients as
$MSQ(G,\boldsymbol{g})$.
\end{itemize}

\begin{rem}
\label{rem:not Inn}In the definition above for a $\boldsymbol{g}$-shortest
homomorphism, it may be assumed that $\pp$ is not generated by elements
of $Inn(L)$, since the effect of right-composition with any element
of $Inn(L)$ can be emulated by left-composition with some element
of $Inn(\FF)$. (This observation will simplify things in Section
\ref{sec:examples}.)
\end{rem}
\begin{rem}
\label{rem:factorisation usefulness}One of the properties of $MSQ(G,\boldsymbol{g})$
is that any $h\in Hom(G,\mathbb{F})$ factors through some $(L,\eta)\in MSQ(G,\boldsymbol{g})$,
i.e. there exist $\varphi\in Mod(G)$ and $h'\in Hom(L,\mathbb{F})$
such that $h=h'\circ\eta\circ\varphi$. We will soon define a subset
$\widetilde{MSQ}(G,\boldsymbol{g})\subseteq MSQ(G,\boldsymbol{g})$
for which this property remains.
\end{rem}
\begin{defn}
\cite{key-7} An \emph{isomorphism of shortening quotients}, or an
\emph{SQ-isomorphism}, is a group isomorphism between two shortening
quotients (not necessarily with respect to the same generating set),
which in addition respects the quotient maps for some modular automorphism
of $G$ (i.e.\linebreak{} $\sigma:M\stackrel[Groups]{\sim}{\rightarrow}L$ for
$(L,\eta)\in SQ(G,\boldsymbol{g}),\,(M,\pi)\in SQ(G,\boldsymbol{u})$,
such that there exists $\pp\in Mod(G)$ for which $\eta\circ\pp=\sigma\circ\pi$.
In other words, the following diagram must commute: $\begin{array}{ccc}
G & \overset{\pp}{\longrightarrow} & G\\
{\scriptstyle \pi}\downarrow &  & \downarrow{\scriptstyle \eta}\\
M & \underset{\sigma}{\longrightarrow} & L
\end{array}$). In particular, if both shortening quotients are $\boldsymbol{g}$-shortening
quotients, then this is an \emph{isomorphism of $\boldsymbol{g}$-shortening
quotients}.
\end{defn}
Notice that it is possible for two $\boldsymbol{g}$-shortening quotients
to be isomorphic as groups but not as shortening quotients. Denote
$SQ(G,\boldsymbol{g})/\sim$ the set of equivalence classes of \emph{$\boldsymbol{g}$}-shortening
quotients of $G$, where $(L,\eta),(M,\pi)\in SQ(G,\boldsymbol{g})$
are equivalent iff they are SQ-isomorphic. (Note that this is indeed
an equivalence relation.) Likewise denote the set of equivalence classes
of maximal $\boldsymbol{g}$-shortening quotients as $MSQ(G,\boldsymbol{g})/\sim$.

\subsubsection*{Properly maximal shortening quotients}

For a limit group $G$ with a given generating set $\boldsymbol{g}$,
let $\widetilde{MSQ}(G,\boldsymbol{g})$ denote the set of \emph{properly
maximal $\boldsymbol{g}$-shortening quotients}, i.e. the subset of
$MSQ(G,\boldsymbol{g})$ consisting of the elements whose equivalence
classes are maximal with respect to the partial order defined on $MSQ(G,\boldsymbol{g})/\sim$
by $[(N,\eta)]\leq[(Q,q)]$ iff there exist an epimorphism $\tau:Q\twoheadrightarrow N$
and some $\pp\in Mod(G)$ such that the following diagram commutes:

\hspace{-0.5cm}$\begin{array}{ccc}
G & \overset{\pp}{\longrightarrow} & G\\
{\scriptstyle q}\downarrow &  & \downarrow{\scriptstyle \eta}\\
Q & \underset{\tau}{\twoheadrightarrow} & N
\end{array}$.

\hspace{-0.5cm}It is easy to check that this relation is well defined
on the equivalence classes. Additionally, this is indeed a partial
order: it is clearly reflexive and transitive. As for antisymmetry,
if $[(N,\eta)]\leq[(Q,q)]$ but also $[(Q,q)]\leq[(N,\eta)]$, then
in particular there exist two epimorphisms $\tau_{1}:Q\twoheadrightarrow N,$
$\tau_{2}:N\twoheadrightarrow Q$. Their composition is an epimorphism
$\tau_{2}\circ\tau_{1}:N\twoheadrightarrow N$, and by the Hopf property
for limit groups (end of Section 4 in \cite{key-7}) it follows that
$\tau_{2}\circ\tau_{1}$ is a group isomorphism. Consequently, also
$\tau_{1}$ is a group isomorphism, so indeed $[(N,\eta)]=[(Q,q)]$.

In a manner of speaking, $\widetilde{MSQ}(G,\boldsymbol{g})$ is sufficient
for the sake of studying $Hom(G,\FF)$, as the property that every
$h\in Hom(G,\FF)$ factors through some element of $MSQ(G,\boldsymbol{g})$
(up to composition with some $\pp\in Mod(G)$), is preserved by the
subset $\widetilde{MSQ}(G,\boldsymbol{g})$.

\subsubsection*{Strict shortening quotients}
\begin{defn}
\cite{key-7} A $\boldsymbol{g}$-shortening quotient $(L,\eta)$
is called a \emph{strict shortening quotient} if:

    \renewcommand{\labelenumi}{(\roman{enumi})}
\begin{enumerate}
\item For every rigid vertex group $G_{v}$ in the cyclic JSJ decomposition
of $G$, replace each neighbouring abelian vertex group by the direct
summand of its edge groups. Let $\tilde{G}_{v}$ be the subgroup of
$G$ generated by $G_{v}$ and by the centralisers of the edge groups
connected to $G_{v}$ in the new graph of groups. Then $\eta\vert_{\tilde{G}_{v}}$is
a monomorphism.
\item For every surface vertex group $S$ in the cyclic JSJ decomposition
of $G$, $\eta(S)$ is non-abelian, and boundary elements of $S$
have non-trivial images.
\item For every abelian vertex group $A$ in the cyclic JSJ decomposition
of $L$, let $\tilde{A}<A$ be the subgroup generated by all edge
groups connected to the vertex stabilised by $A$. Then $\eta\vert_{\tilde{A}}$
is a monomorphism.
\end{enumerate}
\end{defn}
Among the elements of $MSQ(G,\boldsymbol{g})$ there is at least one
strict maximal $\boldsymbol{g}$-shortening quotient (Lemma 5.10 in
\cite{key-7}).

It also worth noting that strictness is a property that is preserved
under SQ-isomorphisms, and consequently a strict shortening quotient
cannot be SQ-isomorphic to a non-strict shortening quotient.

\subsubsection*{Makanin-Razborov diagrams}

The Makanin-Razborov diagram of a limit group $G$ with a finite generating
set $\boldsymbol{g}$ is constructed iteratively, so that each level
is comprised of (restricted) maximal shortening quotients of freely
indecomposable components of groups from the previous level. These
maximal shortening quotients are taken with respect to generating
sets inherited from $\boldsymbol{g}$. For full detail see \cite{key-7}.
Due to the factorisation property of maximal shortening quotients,
noted in Remark \ref{rem:factorisation usefulness}, this diagram
encodes all the elements of $Hom(G,\FF)$.

\subsubsection*{Restricted Makanin-Razborov diagrams}

(cf. \cite{key-7})

Let $2\leq k\in\mathbb{N}$, and fix an ordered generating set $(y_{1},...,y_{k})$
for $\FF_{k}$. For a f.g. group $G$ and a finite ordered subset
$(\gamma_{1},...,\gamma_{k})\subseteq G$ which generates a proper
subgroup $\Gamma<G$, denote\linebreak{}
 $Hom_{\Gamma}(G,\FF_{k})=\{h\in Hom(G,\FF_{k}):\,\forall1\leq i\leq k\,\,h_{n}(\gamma_{i})=y_{i}\}$
the set of \emph{restricted }\linebreak{}
\emph{homomorphisms}. A \emph{restricted limit group relative to $(\gamma_{1},...,\gamma_{k})$}
is a limit group $L=G/\sk h_{n}$ for a sequence $\{h_{n}\}_{\nn}\subseteq Hom_{\Gamma}(G,\FF_{k})$.

Many of the definitions in the non-restricted case can be modified
to suit the restricted case, such as the \emph{relative Grushko decomposition
relative to $\Gamma$}, the canonical \emph{restricted cyclic JSJ
decomposition} of $L$, the\emph{ restricted modular group} $RMod(L)$,
the set of\emph{ restricted $\boldsymbol{g}$-shortening quotients}
and \emph{restricted maximal $\boldsymbol{g}$-shortening quotients},
denoted $RSQ(L,\boldsymbol{g})$ and $RMSQ(L,\boldsymbol{g})$ respectively,
and the \emph{restricted MR diagram with respect to }\textbf{\emph{$\boldsymbol{g}$}}.
It is of particular interest to note that $RMod(L)$ is generated
by the elements that generate $Mod(L)$ and also stabilise $(\eta(\gamma_{1}),...,\eta(\gamma_{k}))$,
and consequently\linebreak{}
$Inn(L)\nsubseteq RMod(L)$.

Just as in the non-restricted case, there exists at least one \emph{ strict
}element of $RMSQ(L,\boldsymbol{g})$.

\subsubsection*{Essential JSJ decompositions}
\begin{defn}
For some limit groups it is possible to modify the cyclic JSJ decomposition
to a canonical essential splitting, called the \emph{essential JSJ
decomposition} (cf. \cite{key-10} and \cite{key-11}). An \emph{essential
$\mathbb{Z}$-splitting} of a group is a splitting whose edge groups
are all maximal cyclic subgroups.
\end{defn}

\subsubsection*{Ivanov words}
\begin{defn}
\cite{key-3,key-2} A \emph{C-test word in n letters} is a non-trivial
word\linebreak{}
$w(x_{1},...,x_{n})\in\FF_{n}=\langle x_{1},...,x_{n}\rangle$ such
that for any f.g. free group $F$ and $n$-tuples\linebreak{}
$(A_{1},...,A_{n}),(B_{1},...,B_{n})\in F^{n}$, if $w(A_{1},...,A_{n})=w(B_{1},...,B_{n})\neq1$
then there exists some $S\in F$ such that $SA_{i}S^{-1}=B_{i}$ for
all $1\leq i\leq n$.

An \emph{Ivanov word} is a $C$-test word in $n$ letters which is
not a proper power, and with the additional property that for elements
$A_{1},...,A_{n}$ in any free group $F$, $w(A_{1},...,A_{n})=1$
iff $\langle A_{1},...,A_{n}\rangle$ is a cyclic subgroup of $F$.
\end{defn}
\begin{lem}
\emph{(cf. \cite{key-12}, Corollary 1)\label{S of Ivanov}} Let $\varphi,\psi\in End(\mathbb{F}_{n})$
such that $\psi$ is a monomorphism and\linebreak{}
$\varphi(w)=\psi(w)$ for $w$ which is an Ivanov word in $n$ letters.
Then $\varphi=\tau_{S}\circ\psi$ for some $S\in\FF_{n}$ such that
$\langle S,\psi(w)\rangle\leq\mathbb{F}_{n}$ is cyclic, where $\tau_{S}$
is the conjugation by $S$.

If, in addition, $\psi$ is surjective, then $S\in\langle\psi(w)\rangle$.
\end{lem}
\begin{proof}
Let $w(x_{1},...,x_{n})$ be an Ivanov word in $n$ letters. Denote
$A_{i}=\psi(x_{i})$, $B_{i}=\pp(x_{i})$ for\linebreak{}
$1\leq i\leq n$. Then $\pp(w(x_{1},...,x_{n}))=w(\pp(x_{1}),...,\pp(x_{n}))=w(B_{1},...,B_{n})$,
and likewise\linebreak{}
$\psi(w(x_{1},...,x_{n}))=w(A_{1},...,A_{n})$. Since $\psi(w(x_{1},...,x_{n}))=\pp(w(x_{1},...,x_{n}))$,
it follows that\linebreak{}
$w(A_{1},...,A_{n})=w(B_{1},...,B_{n})$, and since $\psi$ is a monomorphism
and $w(x_{1},...,x_{n})\neq1_{\mathbb{F}_{n}}$, it follows also that
$w(A_{1},...,A_{n})\neq1_{\FF_{n}}$. Consequently, there exists some
$S\in\FF_{n}$ such that $SA_{i}S^{-1}=B_{i}$ for all $1\leq i\leq n$
(because $w$ is a C-test word). This can also be written as $\tau_{S}\circ\psi(x_{i})=\pp(x_{i})$
for all $1\leq i\leq n$. $(x_{1},...,x_{n})$ is a generating set
of $\FF_{n}$, so in fact $\tau_{S}\circ\psi=\pp$.

In particular, $Sw(A_{1},...,A_{n})S^{-1}=w(B_{1},...,B_{n})=w(A_{1},...,A_{n})$.
In other words, $w(A_{1},...,A_{n})$ and $S$ commute in $\FF_{n}$.
This is only possible if there exists a cyclic subgroup $\langle c\rangle\leq\FF_{n}$
to which both $w(A_{1},...,A_{n})$ and $S$ belong.

Now assume in addition that $\psi$ is surjective, then $S\in Im(\psi)$.
Let $p,q\in\mathbb{Z}$ s.t. $S=c^{p}$ and $\psi(w)=c^{q}$. WLOG
assume $gcd(p,q)=1$ (else take $c^{gcd(p,q)}$ instead of $c$).
By Bézout's lemma, there exist $k,\ell\in\mathbb{Z}$ with $p\cdot k+q\cdot\ell=1$.
It follows that $c=c^{p\cdot k+q\cdot\ell}=c^{p\cdot k}\cdot c^{q\cdot\ell}=S^{k}\cdot(\psi(w))^{\ell}\in Im(\psi)$.
Since $\psi$ is a monomorphism, there exists a unique element $v=\psi^{-1}(c)$.
This element is in fact a root of $w(x_{1},x_{2})$: $\psi(v^{q})=c^{q}=\psi(w(x_{1},x_{2}))$,
but $\psi$ is a monomorphism, so $v^{p}=w(x_{1},x_{2})$.

\end{proof}
Notice that in the lemma above, if $\psi$ is surjective then $S=\psi(w)^{k}$
for some $k\in\mathbb{Z}$, so $\varphi$ can be written as $\varphi=\tau_{\psi(w)}^{k}\circ\psi=\psi\circ\tau_{w}^{k}$.
\begin{lem}
\emph{\cite{key-3,key-12,key-2}}\label{lem:Ivanov word}~%
\noindent\begin{minipage}[t]{1\columnwidth}%
$w(x_{1},x_{2})=[x_{1},x_{2}]^{100}x_{1}[x_{1},x_{2}]^{200}x_{1}[x_{1},x_{2}]^{300}x{}_{1}^{-1}[x_{1},x_{2}]^{400}x_{1}^{-1}\cdot$

\hspace{1.75cm}$\cdot[x_{1},x_{2}]^{500}x_{2}[x_{1},x_{2}]^{600}x_{2}[x_{1},x_{2}]^{700}x_{2}^{-1}[x_{1},x_{2}]^{800}x_{2}^{-1}$ %
\end{minipage}

\hspace{-0.5cm}is an Ivanov word in 2 letters. 

\end{lem}
This particular Ivanov word will be extremely useful in the construction
of the group in Section \ref{sec:Counterexample}.

\subsubsection*{Test sequences}
\begin{defn}
Let $G$ be a freely indecomposable limit group with a strict resolution
\linebreak{}
$G=G_{0}\overset{\eta_{1}}{\twoheadrightarrow}G_{1}\overset{\eta_{2}}{\twoheadrightarrow}...\overset{\eta_{k}}{\twoheadrightarrow}G_{k}\overset{\eta_{k+1}}{\twoheadrightarrow}G_{k+1}=\FF_{\ell}$,
and for every $0\leq i\leq k$ let $\Lambda_{i}$ be a graph of groups
decomposition of $G_{i}$. Order the edges of the decomposition $\{\aij eij\}_{j=1}^{q_{i}}$.
Choose a free generating set $(x_{1},...,x_{\ell})$ for $\FF_{\ell}$.
A \emph{test sequence} of the resolution (with respect to $\boldsymbol{x}$)
is a sequence $\{h_{n}\}_{\nn}\subseteq Hom(G,\FF)$ with the following
properties:
\begin{enumerate}
\item There exist $\{\aij hin\}_{\nn}\subseteq Hom(G_{i},\FF)$ and $\aij{\sigma}{i-1}n\subseteq Mod(G_{i-1})$
for all $1\leq i\leq k+1$ such that $\aij hin=\aij h{i+1}n\circ\eta_{i+1}\circ\aij{\sigma}in$
and $h_{n}=\aij h1n\circ\eta_{i+1}\circ\aij{\sigma}0n$.
\item $\underset{i}{max}\abs{\aij h{k+1}n(x_{i})}=\chi_{n}$, $\underset{i}{min}\abs{\aij h{k+1}n(x_{i})}=\xi_{n}$
with $\frac{\xi_{n}}{\chi_{n}}\rightarrow1$, and cancellation between
$h_{n}^{(k+1)}(x_{i})$ and $h_{n}^{(k+1)}(x_{j})$ for $i\neq j$
is at most $c_{n}$ such that $\frac{c_{n}}{\chi_{n}}\rightarrow0$.
\item $\sk\aij hin=1_{\FF}$, and furthermore the graph of groups associated
with the limit tree of $G_{i}/\sk\aij hin$ is the one-edged splitting
obtained from $\Lambda_{i}$ by collapsing all edges but $\aij ei1$.
\item For every $1\leq r<q_{i}$, the connected component of $\Lambda_{i}\backslash\{\aij eij\}_{j=1}^{r}$
containing $\aij ei{r+1}$ has a fundamental group $\aij G{r+1}i$.
$\sk\aij hin\vert_{\aij G{r+1}i}=1$, and furthermore the graph of
groups associated with the limit tree of $\aij G{r+1}i/\sk(\aij hin\vert_{\aij G{r+1}i})$
is the one-edged splitting obtained from the connected component of
$\Lambda_{i}\backslash\{\aij eij\}_{j=1}^{r}$ containing $\aij ei{r+1}$
by collapsing all edges but $\aij ei{r+1}$.
\end{enumerate}
\end{defn}
Test sequences are known to exist for some types of groups, among
them freely indecomposable limit groups when the associated graphs
of groups are those of the JSJ decompositions in which all the vertices
are rigid. Their properties feature heavily in the arguments presented
in Section \ref{sec:examples}.

\section{Counterexample to canonicality of MR diagrams\label{sec:Counterexample}}

The \textbf{construction} of the (restricted) Makanin-Razborov diagram
of a (restricted) limit group $L$ \linebreak{}
depends on the chosen generating set $\boldsymbol{g}$, for it is
with respect to generating sets inherited from $\boldsymbol{g}$ that
the (restricted) maximal shortening quotients are taken. However,
this does not automatically mean that the\textbf{ resulting diagram}
likewise depends on the choice of generating set.

To examine the canonicality of (restricted) MR-diagrams with dependence
only on the limit group, and not also on the generating set, it is
enough to examine the canonicality of the first level, since all other
levels are built iteratively.

For a (restricted) limit group $L$ with two different generating
sets $\boldsymbol{g}=(g_{1},...,g_{t})$ and\linebreak{}
$\boldsymbol{u}=(u_{1},...,u_{r})$, is the (restricted) MR diagram
of $L$ with respect to $\boldsymbol{g}$ the same as the (restricted)
MR diagram of $L$ with respect to $\boldsymbol{u}$, up to isomorphism
of shortening quotients? In the restricted case the answer is negative,
and this section is devoted to constructing a counterexample.

$ $

For a word $1\neq v\in\FF_{k}$ which is not primitive and has no
roots, the group $\FF_{k}\agm{\langle v\rangle}\FF_{k}$ is a limit
group. S. Heil \cite{key-2}  has described all the possible forms
of cyclic JSJ decompositions of a double of free groups of rank 2
(along a word which is not necessarily Ivanov). Some of those forms
can be eliminated when taking such a double along an Ivanov word $w(a,b)$
for some generating set $\{a,b\}$ of $\FF_{2}$. For example, this
eliminates all forms that are possible iff $w(a,b)\in\langle xyx^{-1},y\rangle$
for some (other) generating set $\{x,y\}$ of $\FF_{2}=\langle a,b\rangle$.
Suppose otherwise, then $w(a,b)\in\langle xyx^{-1},y\rangle$ for
some generating set $\{x,y\}$ of $\FF_{2}$. Consequently, there
exists $\pp\in Aut(\FF_{2})$ which is not an inner automorphism but
fixes $y$ and $xyx^{-1}$ (for example, $\pp$ that is defined by
$\pp(x)=xy,\pp(y)=y$). So this $\pp$ fixes the group generated by
$y$ and $xyx^{-1}$. In particular, $\pp$ fixes $w(a,b)$, so $w(a,b)=\pp\left(w(a,b)\right)=w\left(\pp(a),\pp(b)\right)$.
Since $w$ is an Ivanov word, it follows that $\pp(b)=SbS^{-1}$ and
$\pp(a)=SaS^{-1}$. But this means that $\pp$ is an inner automorphism,
a contradiction. By similar argument, it can be shown that $w$ does
not correspond to a simple closed curve in a surface whose surface
group appears in the JSJ decomposition (because such words can be
fixed by non-inner automorphisms, whereas Ivanov words cannot). The
three remaining cyclic JSJ forms, described in figure \ref{fig:essential JSJ forms},
all share the same essential JSJ form, which coincides with the double
decomposition.

\begin{figure}[tbh]
	\centering
	\includegraphics[width=0.7\linewidth]{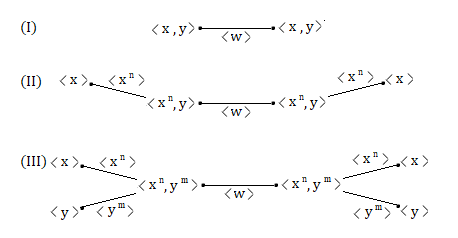}
	\caption{\label{fig:essential JSJ forms}Possible JSJ forms of the double of $\protect\FF_{2}$ along an Ivanov word $w$}
	\label{fig:essential JSJ forms}
\end{figure}

$ $

Let $w(y_{1},y_{2})$ be the Ivanov word given in Lemma \ref{lem:Ivanov word},
and let\linebreak{}
$G_{w}(a_{1},a_{2},b_{1},b_{2})=F(a_{1},a_{2})\agm{w(a_{1},a_{2})=w(b_{1},b_{2})}F(b_{1},b_{2})$
be the double of $\FF_{2}$ over that word. Take the graph of groups
associated with the double $D_{w}=G_{w}(a_{1},a_{2},b_{1},b_{2})\underset{w(b_{1},b_{2})=w(c_{1},c_{2})}{*}G_{w}(d_{1},d_{2},c_{1},c_{2})$,
and add an edge between the two vertices. Let the edge group of the
new edge be the group generated by the element $z=x_{1}y_{2}x_{1}y_{1}x_{1}y_{1}$
in $G_{w}(x_{1},x_{2},y_{1},y_{2})$.

The resulting double-edged double will be denoted $D_{w,z}$ (see
figure \ref{fig:D_w,z}). Notice that by adding the second edge, the
underlying graph of the graph of groups is no longer a tree. Therefore,
$z$ can be embedded by the identity function into only one of the
vertex groups, whereas the embedding into the other vertex group must
be by conjugation with a Bass-Serre element $\gamma$. This gives
the equation $a_{1}b_{2}a_{1}b_{1}a_{1}b_{1}=z(a_{1},a_{2},b_{1},b_{2})=\gamma z(d_{1},d_{2},c_{1},c_{2})\gamma^{-1}=\gamma d_{1}c_{2}d_{1}c_{1}d_{1}c_{1}\gamma^{-1}$.

\begin{figure}[tbh]
	\centering
	\includegraphics[width=0.7\linewidth]{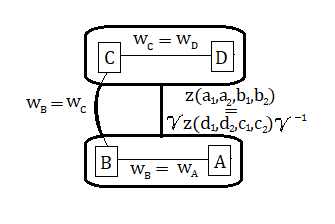}
	\caption{\label{fig:D_w,z}The group $D_{w,z}$. The bold lines represent the essential JSJ decomposition.}
	\label{fig:D_w,z}
\end{figure}

For simplicity, assume the notation $A=F(a_{1},a_{2}),B=F(b_{1},b_{2}),C=F(c_{1},c_{2}),D=F(d_{1},d_{2})$.
\begin{lem}
$D_{w,z}$ is a restricted limit group with respect to the coefficients
$\{b_{1},b_{2}\}$.
\end{lem}
\begin{proof}
To show that $D_{w,z}$ is a restricted limit group, it is enough
to find a restricted strict MR resolution from $D_{w,z}$ to $\FF_{2}=\langle b_{1},b_{2}\rangle$.
(This is due to the modification of Theorem 5.12 in \cite{key-7}
to the restricted case. See Definition 5.11 there of a strict MR resolution
of a f.g. group which is not necessarily a limit group.) The suggested
resolution is $D_{w,z}\overset{\eta}{\rightarrow}G_{w}\overset{\pi}{\rightarrow}\FF_{2}$,
where $\eta(\gamma)=1$, $\eta(c_{i})=\eta(b_{i})=b_{i}$, $\eta(d_{i})=\eta(a_{i})=a_{i}$
and $\pi(a_{i})=\pi(b_{i})=b_{i}$ for $i\in\{1,2\}$. It is clearly
restricted with respect to the coefficients $\{b_{1},b_{2}\}$. It
remains to check that it is indeed a strict MR resolution. There are
four criteria in \cite{key-7} 5.11, examine each in turn.

    \renewcommand{\labelenumi}{(\roman{enumi})}
\begin{enumerate}
\item  BWOC assume that $\langle z(a_{1},a_{2},b_{1},b_{2})\rangle$ is
not maximal abelian in $G_{w}$, then $z(a_{1},a_{2},b_{1},b_{2})=x^{j}$
for some $x\in G_{w}$ and $j\in\mathbb{Z}\backslash\{\pm1\}$. Then
$b_{1}b_{2}b_{1}^{4}=\eta(z(a_{1},a_{2},b_{1},b_{2}))=y^{j}$ for\linebreak{}
$\eta(x)=y\in\FF_{2}=\langle b_{1},b_{2}\rangle$.  This is not possible,
as $b_{1}b_{2}b_{1}^{4}$ clearly has no roots in $\FF_{2}=\langle b_{1},b_{2}\rangle$.
A similar argument can be used for $w$, bearing in mind that $w(b_{1},b_{2})$
is known to have no roots in $\FF_{2}=\langle b_{1},b_{2}\rangle$.
\item $\eta$ is indeed monomorphic on $\langle A\agm{\langle w\rangle}B,\,w,\,z(a_{1},a_{2},b_{1},b_{2})\rangle$
and on $\langle C\agm{\langle w\rangle}D,\,w,\,z(d_{1},d_{2},c_{1},c_{2})\rangle$,
and $\pi$ is monomorphic on $\langle A,\,w\rangle$ and on $\langle B,\,w\rangle$.
\end{enumerate}
There are no QH vertices, nor abelian vertices, in the given splittings,
so criteria (iii) and (iv) hold vacuously.

Hence the suggested resolution is indeed a restricted strict MR resolution
of $D_{w,z}$.

\end{proof}
\begin{lem}
The double-edged double decomposition

\hspace{-0.5cm}$D_{w,z}=G_{w}(a_{1},a_{2},b_{1},b_{2})\stackrel[z(a_{1},a_{2},b_{1},b_{2})=\gamma^{-1}z(d_{1},d_{2},c_{1},c_{2})\gamma]{w(b_{1},b_{2})=w(c_{1},c_{2})}{*}G_{w}(d_{1},d_{2},c_{1},c_{2})$
is also the essential JSJ decomposition of $D_{w,z}$.
\end{lem}
\begin{proof}
The double decomposition of $G_{w}=F(a_{1},a_{2})\underset{w(a_{1},a_{2})=w(b_{1},b_{2})}{*}F(b_{1},b_{2})$
is also the essential JSJ decomposition of $G_{w}$. The double $D_{w}=G_{w}(a_{1},a_{2},b_{1},b_{2})\underset{w(b_{1},b_{2})=w(c_{1},c_{2})}{*}G_{w}(d_{1},d_{2},c_{1},c_{2})$
is a limit group (as seen in the previous lemma), but the double decomposition
of $D_{w}$ is not its essential JSJ decomposition, as both vertex
groups can be further split with respect to the edge group. However,
the double-edged double decomposition 

\hspace{-0.5cm}$D_{w,z}=G_{w}(a_{1},a_{2},b_{1},b_{2})\stackrel[z(a_{1},a_{2},b_{1},b_{2})=\gamma^{-1}z(d_{1},d_{2},c_{1},c_{2})\gamma]{w(b_{1},b_{2})=w(c_{1},c_{2})}{*}G_{w}(d_{1},d_{2},c_{1},c_{2})$
is the essential JSJ decomposition of $D_{w,z}$. This is due to the
fact that the vertex groups are both $G_{w}$, whose only possible
essential splitting is the double decomposition. But this splitting
is not compatible with the incident edges.
\end{proof}
\begin{lem}
\label{lem:Hom form}Let $h\in Hom_{B}(D_{w,z},B)$, then $h$ is
of the form $h\vert_{B}=id_{B}$, $h\vert_{A}=(\tau_{w_{B}})^{k}\circ\{a_{i}\mapsto b_{i}\}$,
$h\vert_{C}=(\tau_{w_{B}})^{\ell}\circ\{c_{i}\mapsto b_{i}\}$, $h\vert_{D}=(\tau_{w_{B}})^{k+\ell}\circ\{d_{i}\mapsto b_{i}\}$
and $h(\gamma)=h(z^{q})w^{-\ell}$ for some $k,\ell,q\in\mathbb{Z}$,
where $\tau_{w_{B}}$ is the conjugation by $w_{B}$.
\end{lem}
\begin{proof}
Let $h\in Hom_{B}(D_{w,z},B)$, so in particular $h(w)\neq1$. $h(A)$
is a 2-generated subgroup of $\mathbb{F}_{2}$, and as such $h(A)\in\left\{ \{1\},\mathbb{Z},\mathbb{F}_{2}\right\} $.
Likewise $h(C),h(D)\in\left\{ \{1\},\mathbb{Z},\mathbb{F}_{2}\right\} $.
But if $h(A)$ is cyclic then $h(w)=1$, a contradiction. Therefore
$h(A)\cong\mathbb{\mathbb{F}}_{2}$, and by similar argument $h(C)\cong\mathbb{\mathbb{F}}_{2}\cong h(D)$
(and by assumption $h(B)=B$). 

$h\vert_{A},h\vert_{B},h\vert_{C},h\vert_{D}\in End(\mathbb{F}_{2})$
all agree on the word $w$ in their respective generating sets, and
are all monomorphic (by the Hopf property for f.g. free groups, since
they are all epimorphisms from $\FF_{2}$ to itself). Because $w$
is an Ivanov word, it follows that 

\hspace{-0.5cm}\textbf{$h\vert_{A}=h\vert_{B}\circ(\tau_{w_{B}})^{k}\circ\{a_{i}\mapsto b_{i}\}=(\tau_{w_{B}})^{k}\circ\{a_{i}\mapsto b_{i}\}$},\textbf{
$h\vert_{C}=h\vert_{B}\circ(\tau_{w_{B}})^{\ell}\circ\{c_{i}\mapsto b_{i}\}=(\tau_{w_{B}})^{\ell}\circ\{c_{i}\mapsto b_{i}\}$}

\hspace{-0.5cm}and\textbf{ $h\vert_{D}=h\vert_{C}\circ(\tau_{w_{C}})^{m}\circ\{d_{i}\mapsto c_{i}\}=(\tau_{w_{B}})^{\ell+m}\circ\{d_{i}\mapsto b_{i}\}$}
for some $k,\ell,m\in\mathbb{Z}$.  The choice $a_{1}b_{2}a_{1}b_{1}a_{1}b_{1}=z=\gamma d_{1}c_{2}d_{1}c_{1}d_{1}c_{1}\gamma^{-1}$
gives rise to the equation

\hspace{-0.5cm}$w^{k}b_{1}w^{-k}b_{2}w^{k}b_{1}w^{-k}b_{1}w^{k}b_{1}w^{-k}b_{1}=h(a_{1}b_{2}a_{1}b_{1}a_{1}b_{1})=h(\gamma d_{1}c_{2}d_{1}c_{1}d_{1}c_{1}\gamma^{-1})=$

$=h(\gamma)w^{\ell+m}b_{1}w^{-(\ell+m)}w^{\ell}b_{2}w^{-\ell}w^{\ell+m}b_{1}w^{-(\ell+m)}w^{\ell}b_{1}w^{-\ell}w^{\ell+m}b_{1}w^{-(\ell+m)}w^{\ell}b_{1}w^{-\ell}(h(\gamma))^{-1}=$

$=\left(h(\gamma)w^{\ell}\right)w^{m}b_{1}w^{-m}b_{2}w^{m}b_{1}w^{-m}b_{1}w^{m}b_{1}w^{-m}b_{1}\left(w^{-\ell}(h(\gamma))^{-1}\right)$.

\hspace{-0.5cm}By taking $\varepsilon(r)=\begin{cases}
\begin{array}{c}
1\\
0
\end{array} & \begin{array}{c}
0<r\\
r\leq0
\end{array}\end{cases}$ , the equation

\hspace{-0.5cm}$\left(b_{1}\right)^{-7\varepsilon(k)}w^{k}b_{1}w^{-k}b_{2}w^{k}b_{1}w^{-k}b_{1}w^{k}b_{1}w^{-k}b_{1}\left(b_{1}\right)^{7\varepsilon(k)}=$

$=\left(\left(b_{1}\right)^{-7\varepsilon(k)}h(\gamma)w^{\ell}\left(b_{1}\right)^{7\varepsilon(m)}\right)\left(b_{1}\right)^{-7\varepsilon(m)}w^{m}b_{1}w^{-m}b_{2}w^{m}b_{1}w^{-m}b_{1}w^{m}b_{1}w^{-m}b_{1}\left(b_{1}\right)^{7\varepsilon(m)}\cdot$

$\,\,\cdot\left(\left(b_{1}\right)^{-7\varepsilon(m)}w^{-\ell}(h(\gamma))^{-1}\left(b_{1}\right)^{7\varepsilon(k)}\right)$

\hspace{-0.5cm} is between a cyclically reduced word and a conjugation
of a cyclically reduced word in $\FF_{2}$. The only solutions are
when the first cyclically reduced word is equal to some cyclic permutation
of the second cyclically reduced word. This ensures that $m=k$; in
addition the conjugating element must commute with the cyclically
reduced word, hence $h(\gamma)\in\langle h(z)\rangle w^{-\ell}$.
\end{proof}
\begin{rem}
\label{rem:k_n const}Since $D_{w,z}$ has no free decomposition,
the restriction with respect to the set of coefficients $\{b_{1},b_{2}\}$
ensures that the shortening involves neither left-composition with
elements of $Inn(\FF_{2})$, nor right-composition with elements of
$Mod(D_{w,z})$ that do not fix $\{b_{1},b_{2}\}$. Thus, the only
way to shorten $h\in Hom_{B}(D_{w,z},\mathbb{F}_{2})$ is by right-composition
with some power of the Dehn twist along $z$, which affects the value
of $q$, or along $w$, which affects the value of $\ell$ (but not
of $k$).
\end{rem}
\begin{lem}
Let $\boldsymbol{g}$ be a generating set of $D_{w,z}$. Let $\{h_{n}\}_{\nn}\subseteq Hom_{B}(D_{w,z},\mathbb{F}_{2})$
be a sequence of $\boldsymbol{g}$-shortest morphisms, with $k_{n},\ell_{n},q_{n}\in\mathbb{Z}$
as in the previous lemma for each $h_{n}$. If the associated restricted
$\boldsymbol{g}$-shortening quotient  is strict, then $\abs{k_{n}}\underset{n\rightarrow\infty}{\rightarrow}\infty$.
\end{lem}
\begin{proof}
Assume otherwise, then by extraction of subsequence, WLOG the original
sequence, $\{k_{n}\}_{\nn}$ is the constant sequence $k_{n}=k_{0}$.
Hence, for any word $v(x_{1},x_{2})\in F(x_{1},x_{2})$, the $h_{n}$-image
of the element $v(a_{1},a_{2})\in A$ is identified with the $h_{n}$-image
of $w_{B}^{\,k_{0}}v(b_{1},b_{2})w_{B}^{-k_{0}}$ for all $n\in\mathbb{N}$.
Therefore also $\tilde{\eta}(v(a_{1},a_{2}))=w^{k_{0}}\tilde{\eta}(v(b_{1},b_{2}))w^{-k_{0}}=\tilde{\eta}(w^{k_{0}}v(b_{1},b_{2})w^{-k_{0}})$,
contradicting strictness.
\end{proof}

We now examine the set of restricted strict shortening quotients of
$D_{w,z}$ with respect to two\linebreak{}
 generating sets: $\boldsymbol{g}=(a_{1},a_{2},b_{1},b_{2},c_{1},c_{2},d_{1},d_{2},\gamma,a_{1}d_{1}a_{1}d_{1})$
and $\boldsymbol{u}=(a_{1},a_{2},b_{1},b_{2},c_{1},c_{2},d_{1},d_{2},\gamma,a_{1}c_{1}a_{1}c_{1})$.

$ $
\begin{prop}
Let $(\tilde{L},\tilde{\eta})$ be a restricted strict $\boldsymbol{g}$-shortening
quotient of $D_{w,z}$ and $(\tilde{M},\tilde{\pi})$ be a restricted
strict $\boldsymbol{u}$-shortening quotient of $D_{w,z}$. So:
\begin{enumerate}
\item $\tilde{L}=G_{w}$, $\tilde{\eta}\vert_{A\agm{\langle w\rangle}B}=id$,
$\tilde{\eta}\vert_{C\agm{\langle w\rangle}D}$ sends $c_{1},c_{2},d_{1},d_{2}$
to $b_{1},b_{2},a_{1},a_{2}$ respectively (up to conjugation by some
bounded power $\ell$ of $w$), and $\tilde{\eta}(\gamma)=1_{G_{w}}$
(up to multiplication by $w^{-\ell}$).
\item $\tilde{M}$ is $\langle a_{1},a_{2},b_{1},b_{2},\gamma\vert\gamma^{2}a_{i}\gamma^{-2}=w^{\varepsilon}b_{i}w^{-\epsilon},\,i=1,2\rangle$
for some $\varepsilon$, or a quotient thereof.
\end{enumerate}
\end{prop}
\begin{proof}
First consider the generating set $\boldsymbol{g}=(a_{1},a_{2},b_{1},b_{2},c_{1},c_{2},d_{1},d_{2},\gamma,a_{1}d_{1}a_{1}d_{1})$.
Let $(\tilde{L},\tilde{\eta})$ be a restricted strict $\boldsymbol{g}$-shortening
quotient of $D_{w,z}$ and let $\{h_{n}\}_{n\in\mathbb{N}}\subseteq Hom_{B}(D_{w,z},\FF_{2})$
be a sequence of restricted $\boldsymbol{g}$-shortest morphisms with
$Ker\tilde{\eta}=\sk h_{n}$ and with $k_{n},\ell_{n},q_{n}\in\mathbb{Z}$
as above.

For each $n\in\mathbb{N}$, $\underset{g\in\boldsymbol{g}}{max}\vert h_{n}(g)\vert=max\left\{ \begin{array}{c}
\vert b_{i}\vert,\vert w^{k_{n}}b_{i}w^{-k_{n}}\vert,\vert w^{\ell_{n}}b_{i}w^{-\ell_{n}}\vert,\\
\vert w^{\ell_{n}+k_{n}}b_{i}w^{-(\ell_{n}+k_{n})}\vert,\vert h_{n}(z^{q_{n}})w^{-\ell_{n}}\vert,\\
\vert w^{k_{n}}b_{1}w^{\ell_{n}}b_{1}w^{-\ell_{n}}b_{1}w^{\ell_{n}}b_{1}w^{-(k_{n}+\ell_{n})}\vert
\end{array}:i\in\{1,2\}\right\} $. It will be helpful to understand the asymptotic behaviour of these
distances (as $n\rightarrow\infty$) after normalisation by $tr(w^{k_{n}})$.
(Since the free group $B$ acts freely on its Cayley graph, it follows
that $tr(w)\neq0$ and likewise $tr(w^{k_{n}})\neq0$.)

$ $

Notice that $\abs{w^{t}}-tr(w^{t})=const$ for any $t\in\mathbb{Z}$,
and in particular $\frac{\abs{w^{k_{n}}}}{tr(w^{k_{n}})}=1+\frac{O(1)}{tr(w^{k_{n}})}\underset{n\rightarrow\infty}{\rightarrow}1$.\linebreak{}
 It also follows that $\frac{\abs{w^{\ell_{n}}}}{tr(w^{k_{n}})}-\frac{\abs{\ell_{n}}}{\abs{k_{n}}}=\frac{\abs{w^{\ell_{n}}}}{tr(w^{k_{n}})}-\frac{tr(w^{\ell_{n}})}{tr(w^{k_{n}})}=0+\frac{O(1)}{tr(w^{k_{n}})}\underset{n\rightarrow\infty}{\rightarrow}0$.
With the approximation of $\frac{\abs{h_{n}(z^{q_{n}})w^{-l_{n}}}}{tr(w^{k_{n}})}$
there is more need for care, because of the contribution of $q_{n}$.

$\abs{h_{n}(z)}=6tr(w)\cdot\abs{k_{n}}+s_{n}$, where (for large enough
$k_{n}$) $s_{n}\in\mathbb{R}$ is in fact dependent only on the sign
of $k_{n}$. $\abs{h_{n}(z^{q_{n}})}=\abs{h_{n}(z)}\cdot q_{n}$,
because there is no additional cancellation, as $h_{n}(z)$ is cyclically
reduced. If $q_{n}\geq0$ then $w^{-\ell_{n}}$ has no cancellations
with $h_{n}(z^{q_{n}})$ in the multiplication $h_{n}(z^{q_{n}})w^{-\ell_{n}}$,
so $\abs{h_{n}(z^{q_{n}})w^{-\ell_{n}}}=\abs{h_{n}(z)}\cdot q_{n}+\abs{w^{-\ell_{n}}}$.
If $q_{n}<0$, then $h_{n}(z^{q_{n}})w^{-\ell_{n}}=h_{n}(z^{q_{n}+1})\cdot\left(h_{n}(z^{-1})w^{-\ell_{n}}\right)$,
so $\abs{h_{n}(z^{q_{n}})w^{-\ell_{n}}}=\abs{h_{n}(z)}\cdot(\abs{q_{n}}-1)+\left(5tr(w)\cdot\abs{k_{n}}+tr(w)\cdot\abs{k_{n}+\ell_{n}}+s'_{n}\right)$
for $s'_{n}$ whose value depends on the sign of $(k_{n}+\ell_{n})$.
Using $\delta_{n}=\begin{cases}
1 & q_{n}<0\\
0 & 0\leq q_{n}
\end{cases}$ and combining both cases, it follows that $\frac{\abs{h_{n}(z^{q_{n}})w^{-l_{n}}}}{tr(w^{k_{n}})}-\left((6+\frac{s_{n}}{\abs{k_{n}}\cdot tr(w)})\cdot(\abs{q_{n}}-\delta_{n})+5\delta_{n}+\abs{\delta_{n}+\frac{\ell_{n}}{k_{n}}}+\frac{\delta_{n}s'_{n}}{\abs{k_{n}}\cdot tr(w)}\right)=0+\frac{O(1)}{tr(w^{k_{n}})}\underset{n\rightarrow\infty}{\rightarrow}0$.
(Notice that $\frac{s_{n}}{\abs{k_{n}}}=\frac{O(1)}{\abs{k_{n}}}\underset{n\rightarrow\infty}{\rightarrow}0$,
but $\frac{s_{n}}{\abs{k_{n}}}\cdot(\abs{q_{n}}-\delta_{n})$ cannot
be similarly dismissed. However, the term $(6+\frac{s_{n}}{\abs{k_{n}}\cdot tr(w)})\cdot(\abs{q_{n}}-\delta_{n})+5\delta_{n}$
can be lower-bounded by $5\cdot(\abs{q_{n}}-\delta_{n})+5\delta_{n}=5\cdot\abs{q_{n}}$.
Therefore $\frac{\abs{h_{n}(z^{q_{n}})w^{-l_{n}}}}{tr(w^{k_{n}})}\geq5\cdot\abs{q_{n}}+\abs{\delta_{n}+\frac{\ell_{n}}{k_{n}}}+\frac{O(1)}{tr(w^{k_{n}})}\geq4\cdot\abs{q_{n}}$,
and this bound will be useful in analysing the case $q_{n}\neq0$.)

$ $

So the distances can indeed be estimated in units of $tr(w^{k_{n}})=\vert k_{n}\vert\cdot\abs{tr(w)}$.
First assume $q_{n}=0$. 

\hspace{-0.5cm}$\underset{g\in\boldsymbol{g}}{max}\frac{\vert h_{n}(g)\vert}{tr(w^{k_{n}})}=max\left\{ \begin{array}{c}
\frac{O(1)}{tr(w^{k_{n}})},\,2+\frac{O(1)}{tr(w^{k_{n}})},\,2\frac{\vert\ell_{n}\vert}{\vert k_{n}\vert}+\frac{O(1)}{tr(w^{k_{n}})}\\
2\frac{\vert\ell_{n}+k_{n}\vert}{\abs{k_{n}}}+\frac{O(1)}{tr(w^{k_{n}})},\,\frac{\vert\ell_{n}\vert}{\vert k_{n}\vert}+\frac{O(1)}{tr(w^{k_{n}})},\\
1+3\frac{\vert\ell_{n}\vert}{\vert k_{n}\vert}+\frac{\vert k_{n}+\ell_{n}\vert}{\vert k_{n}\vert}+\frac{O(1)}{tr(w^{k_{n}})}
\end{array}:i\in\{1,2\}\right\} $. As mentioned in Remark \ref{rem:k_n const}, $k_{n}$ is given
and cannot be changed by $RMod(D_{w,z})$, but $\ell_{n}$ can be
changed by $RMod(D_{w,z})$. Since $\{h_{n}\}_{\nn}$ are restricted
$\boldsymbol{g}$-shortest homomorphisms, $\ell_{n}$ must be such
that $\underset{g\in\boldsymbol{g}}{max}\frac{\vert h_{n}(g)\vert}{tr(w^{k_{n}})}$
is minimal. Denote $x_{n}=\frac{\ell_{n}}{k_{n}}$, so $\ell_{n}$
must be chosen such that $x_{n}=\underset{\{x=j/k_{n}:j\in\mathbb{Z}\}}{argmin}\,max\left\{ 2,2\vert x\vert,2\vert x+1\vert,\vert x\vert,1+3\vert x\vert+\vert1+x\vert\right\} $.
But $\underset{x\in\mathbb{R}}{min}\text{ }max\left\{ 2,2\vert x\vert,2\vert x+1\vert,\vert x\vert,1+3\vert x\vert+\vert1+x\vert\right\} =2$
is realised at $x_{0}=0$, and therefore also $x_{n}=0+\frac{O(1)}{tr(w^{k_{n}})}$,
i.e. $\ell_{n}=0+O(1)$. By taking $q_{n}\neq0$, $\vert h_{n}(\gamma)\vert$
would raise the value of $\underset{g\in\boldsymbol{g}}{max}\vert h_{n}(g)\vert$
(for in this case $\frac{\abs{h_{n}(z^{q_{n}})w^{-l_{n}}}}{tr(w^{k_{n}})}>4\cdot\abs{q_{n}}\geq4$,
which already exceeds the minimal value 2 which is obtained in the
case $q_{n}=0$), hence $q_{n}=0$ . It follows that the only restricted
strict $\boldsymbol{g}$-shortening quotient of $D_{w,z}$ is $(G_{w},\tilde{\eta})$
where $\tilde{\eta}\vert_{A\agm{\langle w\rangle}B}=id$, $\tilde{\eta}\vert_{C\agm{\langle w\rangle}D}$
sends $c_{1},c_{2},d_{1},d_{2}$ to $b_{1},b_{2},a_{1},a_{2}$ respectively
(up to conjugation by some bounded power $\ell$ of $w$), and $\tilde{\eta}(\gamma)=1_{G_{w}}$
(up to multiplication by $w^{-\ell}$).

$ $

Next consider the generating set $\boldsymbol{u}=(a_{1},a_{2},b_{1},b_{2},c_{1},c_{2},d_{1},d_{2},\gamma,a_{1}c_{1}a_{1}c_{1})$.
Let $(\tilde{M},\tilde{\pi})$ be a restricted strict $\boldsymbol{u}$-shortening
quotient of $D_{w,z}$ and let $\{h_{n}\}_{n\in\mathbb{N}}\subseteq Hom_{B}(D_{w,z},\FF_{2})$
be a sequence of restricted $\boldsymbol{u}$-shortest morphisms with
$Ker\tilde{\pi}=\sk h_{n}$ and with $k_{n},\ell_{n},q_{n}\in\mathbb{Z}$
as in Lemma \ref{lem:Hom form}. By similar analysis, while initially
assuming $q_{n}=0$, ensuring $h_{n}$ is restricted $\boldsymbol{u}$-shortest
means finding $x_{n}=\frac{\ell_{n}}{k_{n}}$ which equals $\underset{\{x=j/k_{n}:j\in\mathbb{Z}\}}{argmin}\text{ }max\left\{ 2,2\vert x\vert,2\vert x+1\vert,\vert x\vert,1+\vert x\vert+3\vert x-1\vert\right\} $.

\hspace{-0.5cm}$\underset{x\in\mathbb{R}}{min}\text{ }max\left\{ 2,2\vert x\vert,2\vert x+1\vert,\vert x\vert,1+\vert x\vert+3\vert x-1\vert\right\} =3$,
and this value is realised at $x=\frac{1}{2}$,\linebreak{}
so $\ell_{n}\in\{\frac{k_{n}}{2}+O(1),\frac{k_{n}+1}{2}+O(1),\frac{k_{n}-1}{2}+O(1)\}$
(since $max\left\{ 2,2\vert x\vert,2\vert x+1\vert,\vert x\vert,1+\vert x\vert+3\vert x-1\vert\right\} $
is monotonically decreasing before $x=\frac{1}{2}$ and monotonically
increasing afterwards).\linebreak{}
Again, taking $q_{n}\neq0$ raises the value of $\underset{u\in\boldsymbol{u}}{max}\vert h_{n}(u)\vert$,
so indeed $q_{n}=0$. For every $n\in\mathbb{N}$ and\linebreak{}
every $1\leq i\leq2$, $h_{n}(\gamma c_{i}\gamma^{-1})=w^{-\ell_{n}}w^{\ell_{n}}b_{i}w^{-\ell_{n}}w^{\ell_{n}}=b_{i}=h_{n}(b_{i})$,
and \linebreak{}
$h_{n}(\gamma d_{i}\gamma^{-1})=w^{-\ell_{n}}w^{k_{n}+\ell_{n}}b_{i}w^{-k_{n}-\ell_{n}}w^{\ell_{n}}=w^{k_{n}}b_{i}w^{-k_{n}}=h_{n}(a_{i})$.
Denote $\varepsilon_{n}=k_{n}-2\ell_{n}$, so $\varepsilon_{n}=O(1)$
. By extraction of subsequence, $\varepsilon_{n}$ is a constant
sequence $\varepsilon$ and $h_{n}(\gamma^{2}a_{i}\gamma^{-2})=h_{n}(w^{\varepsilon}b_{i}w^{-\varepsilon})$.
The limit group $\tilde{M}=G/\sk h_{n}$ is therefore $\langle a_{1},a_{2},b_{1},b_{2},\gamma\vert\gamma^{2}a_{i}\gamma^{-2}=w^{\varepsilon}b_{i}w^{-\epsilon},\,i=1,2\rangle$
or a quotient thereof.
\end{proof}
\begin{figure}[tbh]
	\centering
	\includegraphics[width=0.7\linewidth]{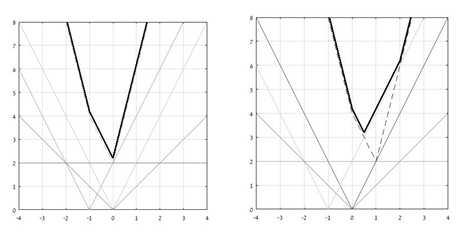}
	\caption{$max\left\{ 2,2\vert x\vert,2\vert x+1\vert,\vert x\vert,1+3\vert x\vert+\vert1+x\vert\right\} $ for $\boldsymbol{g}$ (left), and $max\left\{ 2,2\vert x\vert,2\vert x+1\vert,\vert x\vert,1+\vert x\vert+3\vert x-1\vert\right\} $ for $\boldsymbol{u}$ (right)}
	\label{fig:minimaxrepresentations}
\end{figure}

$ $

For every generating set, there exists a restricted strict maximal
shortening quotient. As seen above,  $(G_{w},\text{\ensuremath{\tilde{\eta}}})$
is the sole strict element of $RSQ(D_{w,z},\boldsymbol{g})$, and
is therefore the only strict element of $RMSQ(D_{w,z},\boldsymbol{g})$.
Likewise, for $\boldsymbol{u}$, any strict restricted shortening
quotient, and in particular any strict restricted maximal shortening
quotient $(\tilde{M},\tilde{\pi})$, is a (possibly not proper) quotient
of $\langle a_{1},a_{2},b_{1},b_{2},\gamma\vert\gamma^{2}a_{i}\gamma^{-2}=w^{\varepsilon}b_{i}w^{-\epsilon},\,i=1,2\rangle$
for some $\varepsilon$. However, no such group  can be isomorphic
to $(G_{w},\tilde{\eta})$; for example, the homology group of $(\tilde{M},\tilde{\pi})$
is $\mathbb{Z}^{t}$, $t\leq3$, whereas the homology group of $(G_{w},\tilde{\eta})$
is $\mathbb{Z}^{4}$. Since $(G_{w},\tilde{\eta})$ and $(\tilde{M},\tilde{\pi})$
are not isomorphic, they are in particular not SQ-isomorphic. A strict
restricted shortening quotient cannot be SQ-isomorphic to a non-strict
shortening quotient, so in fact the only restricted strict $\boldsymbol{g}$-shortening
quotient is not SQ-isomorphic to any restricted maximal $\boldsymbol{u}$-shortening
quotient. 

\section{Generator-independent shortening quotients\label{sec:examples}}

It is worth noting the following observation regarding properly maximal
$\boldsymbol{g}$-shortening quotients:
\begin{lem}
For a limit group $G$, TFAE:\label{lem:TFAE, section 4}
\begin{enumerate}
\item $\widetilde{MSQ}(G,\boldsymbol{g})=\widetilde{MSQ}(G,\boldsymbol{u})$
(up to isomorphism of shortening quotients between the elements of
both sets) for any two generating sets $\boldsymbol{g},\boldsymbol{u}$
of $G$. 
\item $\widetilde{MSQ}(G,\boldsymbol{g})\subseteq SQ(G,\boldsymbol{u})$
(up to SQ-isomorphism of the elements) for any two generating sets
$\boldsymbol{g},\boldsymbol{u}$ of $G$.
\end{enumerate}
\end{lem}
\begin{proof}
The first direction is trivial. In the other direction, let $\boldsymbol{g},\boldsymbol{u}$
be two generating sets of $G$, and let $(Q_{\boldsymbol{g}},q_{\boldsymbol{g}})\in\widetilde{MSQ}(G,\boldsymbol{g})$.
By assumption, for every element of $\widetilde{MSQ}(G,\boldsymbol{g})$
there exists an element of $SQ(G,\boldsymbol{u})$ which is SQ-isomorphic
to it. So there exist $(Q_{\boldsymbol{u}},q_{\boldsymbol{u}})\in SQ(G,\boldsymbol{u})$,
a group isomorphism $\sigma_{1}:Q_{\boldsymbol{u}}\rightarrow Q_{\boldsymbol{g}}$
and $\pp_{1}\in Mod(G)$ such that $q_{\boldsymbol{g}}\circ\pp_{1}=\sigma_{1}\circ q_{\boldsymbol{u}}$.
Since $(Q_{\boldsymbol{u}},q_{\boldsymbol{u}})\in SQ(G,\boldsymbol{u})$,
there exist some $(M_{\boldsymbol{u}},\mu_{\boldsymbol{u}})\in MSQ(G,\boldsymbol{u})$
and an epimorphism $\sigma_{2}:M_{\boldsymbol{u}}\twoheadrightarrow Q_{\boldsymbol{u}}$
such that $q_{\boldsymbol{u}}=\sigma_{2}\circ\mu_{\boldsymbol{u}}$.
There exists some maximal element $(\tilde{M}_{\boldsymbol{u}},\tilde{\mu}_{\boldsymbol{u}})\in\widetilde{MSQ}(G,\boldsymbol{u})$
with an epimorphism $\sigma_{3}:\tilde{M}_{\boldsymbol{u}}\twoheadrightarrow M_{\boldsymbol{u}}$
and some $\pp_{3}\in Mod(G)$ such that $\mu_{\boldsymbol{u}}\circ\pp_{3}=\sigma_{3}\circ\tilde{\mu}_{\boldsymbol{u}}$.
We get the following commutative diagram:

\hspace{-0.5cm}$\begin{array}{ccc}
G & \overset{\pp_{1}\circ\pp_{3}=\pp}{\longrightarrow} & G\\
{\scriptstyle \tilde{\mu}_{\boldsymbol{u}}}\downarrow &  & \downarrow{\scriptstyle q_{\boldsymbol{g}}}\\
\tilde{M}_{\boldsymbol{u}} & \underset{\sigma_{1}\circ\sigma_{2}\circ\sigma_{3}=\sigma}{\twoheadrightarrow} & Q_{\boldsymbol{g}}
\end{array}$ 

\hspace{-0.5cm}Notice that by assumption also $\widetilde{MSQ}(G,\boldsymbol{u})\subseteq SQ(G,\boldsymbol{g})$,
so by symmetric argument there exist $(N_{\boldsymbol{g}},\eta_{\boldsymbol{g}})\in\widetilde{MSQ}(G,\boldsymbol{g})$,
$\psi\in Mod(G)$ and an epimorphism $\tau:N_{\boldsymbol{g}}\twoheadrightarrow\tilde{M}_{\boldsymbol{u}}$
such that $\tilde{\mu}_{\boldsymbol{u}}\circ\psi=\tau\circ\eta_{\boldsymbol{g}}$.

\hspace{-0.5cm}By adding this information to the previous diagram,
the resulting commutative diagram

\hspace{-0.5cm}$\begin{array}{ccc}
G & \overset{\pp\circ\psi}{\longrightarrow} & G\\
{\scriptstyle \eta_{\boldsymbol{g}}}\downarrow &  & \downarrow{\scriptstyle q_{\boldsymbol{g}}}\\
N_{\boldsymbol{g}} & \underset{\sigma\circ\tau}{\twoheadrightarrow} & Q_{\boldsymbol{g}}
\end{array}$

\hspace{-0.5cm}shows that $[(Q_{\boldsymbol{g}},q_{\boldsymbol{g}})]\leq[(N_{\boldsymbol{g}},\eta_{\boldsymbol{g}})]$.
Since $(Q_{\boldsymbol{g}},q_{\boldsymbol{g}})\in\widetilde{MSQ}(G,\boldsymbol{g})$,
$[(Q_{\boldsymbol{g}},q_{\boldsymbol{g}})]$ is maximal in this \linebreak{}
partial order, so $[(Q_{\boldsymbol{g}},q_{\boldsymbol{g}})]=[(N_{\boldsymbol{g}},\eta_{\boldsymbol{g}})]$.
In particular, $\sigma$ is a group isomorphism, and by the first
diagram $(Q_{\boldsymbol{u}},q_{\boldsymbol{u}})$ and $(\tilde{M}_{\boldsymbol{u}},\tilde{\mu}_{\boldsymbol{u}})$
are SQ-isomorphic. Therefore $\widetilde{MSQ}(G,\boldsymbol{g})\subseteq\widetilde{MSQ}(G,\boldsymbol{u})$
(up to SQ-isomorphism of the elements). By symmetric argument $\widetilde{MSQ}(G,\boldsymbol{g})\supseteq\widetilde{MSQ}(G,\boldsymbol{u})$
(up to SQ-isomorphism of the elements), hence the equality.
\end{proof}

\begin{rem*}
The same argument holds with reduction to strict maximal shortening
quotients, since (using the above notation) strictness of $(Q_{\boldsymbol{g}},q_{\boldsymbol{g}})$
passes to $(Q_{\boldsymbol{u}},q_{\boldsymbol{u}})$, so $(M_{\boldsymbol{u}},\mu_{\boldsymbol{u}})$
can be chosen from among the strict elements of $MSQ(G,\boldsymbol{u})$.
The strictness passes on to $(\tilde{M}_{\boldsymbol{u}},\tilde{\mu}_{\boldsymbol{u}})$,
and by symmetric argument also $(N_{\boldsymbol{g}},\eta_{\boldsymbol{g}})$
is strict. Likewise, the argument holds with reduction to restricted
shortening quotients.
\end{rem*}
$ $

As seen in the previous section, in general it cannot be assumed that
the set of (restricted) strict properly maximal shortening quotients
is independent of the generating set. By the lemma above, for a general
limit group $G$ with generating sets $\boldsymbol{g}$ and\textbf{
$\boldsymbol{u}$}, it cannot be assumed that the strict elements
of $\widetilde{MSQ}(G,\boldsymbol{g})$ are in $SQ(G,\boldsymbol{u})$.
However, there are some special cases in which it can be shown that
a strict maximal shortening quotient with respect to one generating
set is a shortening quotient with respect to any other generating
set. Some such cases will be described in this section. The recurring
tool which will be used is the test sequence.

\subsection{First case - strict maximal shortening quotient which is a free group}

Let $L$ be a limit group whose JSJ decomposition is $A\agm{\langle z\rangle}B$
such that $A$ and $B$ are both rigid vertex groups and $\langle z\rangle$
is maximal abelian in both vertex groups. Suppose $\boldsymbol{g}$
is a finite generating set of $L$ and $(\mathbb{F}_{\ell},\eta)$
is a strict maximal $\boldsymbol{g}$-shortening quotient of $L$
( for $2\leq\ell$). Let $\boldsymbol{u}$ be another finite generating
set of $L$. The aim is to find a $\boldsymbol{u}$-shortening quotient
which is SQ-isomorphic to $(\FF_{\ell},\eta)$.

Choose a set of free generators $\boldsymbol{x}=(x_{1},...,x_{\ell})$
of $\FF_{\ell}$. Let $X(\FF_{\ell})$ be the corresponding Cayley
graph, and let $\{h_{n}\}_{\nn}\subseteq Hom(\mathbb{F}_{\ell},\mathbb{F})$
be a test sequence  such that property 2  of the test sequence is
fulfilled with respect to $\boldsymbol{x}$. In particular, this means
that for any $f\in\FF_{\ell}$, on the one hand $\abs{h_{n}(f)}\leq\abs f_{X(\FF_{\ell})}\cdot\chi_{n}$
(by ignoring the possible cancellations), and on the other hand, $\abs{h_{n}(f)}\geq\abs f_{\cayfl}\cdot(\xi_{n}-2c_{n})$
(because of the upper bound on the small cancellation, see figure
\ref{fig:case 1-1}).

\begin{figure}[tbh]
	\centering
	\includegraphics[width=0.7\linewidth]{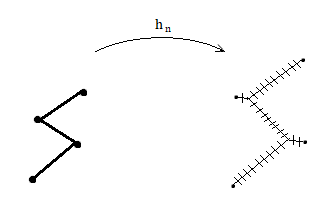}
	\caption{\label{fig:case 1-1}a path in $\protect\cayfl$ and its $h_{n}$-image in $X(\protect\FF)$}
	\label{fig:case 1-1}
\end{figure}

 Let $\{\varphi_{n}\}_{\nn}\subseteq Mod(L)$ and $\{\iota_{n}\}_{\nn}\subseteq Inn(\FF)$
such that $\{\iota_{n}\circ\left(h_{n}\circ\eta\right)\circ\pp_{n}\}_{\nn}$
are $\boldsymbol{u}$-shortest. $Mod(L)$ is generated by the elements
of $Inn(L)$ and by the Dehn twist $\tau_{z}(f)=\begin{cases}
\begin{array}{c}
f\\
zfz^{-1}
\end{array} & \begin{array}{c}
f\in A\\
f\in B
\end{array}\end{cases}$. By Remark \ref{rem:not Inn} it may be assumed that there exist
$\{k_{n}\}_{\nn}\subseteq\mathbb{Z}$ such that $\pp_{n}=\tau_{z}^{k_{n}}$.We
will show that the sequence $\{\pp_{n}\}_{\nn}$ is comprised of a
finite set of homomorphisms, by proving that $\{\abs{k_{n}}\}_{\nn}$
is a bounded set.

First, it is worth noticing the following facts regarding translation
lengths:
\begin{itemize}
\item The translation length of an element $f\in G$ in a Bass-Serre tree
$T$ of $G$ is also given by\linebreak{}
$tr_{T}(f)=\underset{n\rightarrow\infty}{lim}\,\frac{\abs{f^{n}}_{T}}{n}$
.
\item For $f\in\FF_{\ell}$, the translation length of $h_{n}(f)\in\FF$
can be estimated in terms of $tr_{X(\FF_{\ell})}$ similarly to the
way that displacement length in $X(\FF)$ can be estimated in terms
of displacement length in $X(\FF_{\ell})$: $tr_{X(\FF)}(h_{n}(f))=\underset{t\rightarrow\infty}{lim}\,\frac{\abs{h_{n}(f^{t})}}{t}\leq\underset{t\rightarrow\infty}{lim}\,\frac{\chi_{n}\cdot\abs{f^{t}}}{t}=\chi_{n}\cdot tr_{X(\FF_{\ell})}$,
and likewise\linebreak{}
$tr_{X(\FF)}(h_{n}(f))\geq(\xi_{n}-2c_{n})\cdot tr_{X(\FF_{\ell})}$.
\item Translation lengths of conjugate elements are equal.
\end{itemize}
Let $v=ab$ such that $a\in A\backslash\langle z\rangle$ and $b\in B\backslash\langle z\rangle$.
Since $\langle z\rangle$ is maximal abelian in both vertex groups,
this choice ensures that neither $a$ nor $b$ commute with $z$.
Since $1\neq[a,z]\in A$, $1\neq[b,z]\in B$ and $\eta$ is injective
on $A$ and on $B$, it follows that $\eta([a,z]),\,\eta([b,z])\neq1$
and therefore also $\eta(a)$ and $\eta(b)$ do not commute with $\eta(z)$.
Consequently, $\eta(a)$ and $\eta(b)$ do not have the same axis
as $\eta(z)$. So for large enough $N\in\mathbb{N}$, i.e. when $tr_{\FF_{\ell}}(\eta(z^{\pm N}))>\abs{\eta(a)}_{X(\FF_{\ell})}+\abs{\eta(b)}_{X(\FF_{\ell})}$,
it follows that\linebreak{}
\hspace{-0.5cm}$\abs{\eta(\tau_{z}^{\pm N}(v))}_{X(\FF_{\ell})}\geq2N\cdot tr_{\FF_{\ell}}(\eta(z))-(\abs{\eta(a)}_{X(\FF_{\ell})}+\abs{\eta(b)}_{X(\FF_{\ell})})$.
In fact, this lower bound holds also for the translation length: 

\hspace{-0.5cm}$tr_{\FF_{\ell}}(\eta(\tau_{z}^{\pm N}(v)))=\underset{t\rightarrow\infty}{lim}\,\frac{\abs{\eta\circ\tau_{z}^{\pm N}(v)^{t})}}{t}\geq\underset{t\rightarrow\infty}{lim}\,\frac{tr_{X(\FF_{\ell})}(\eta(z))\cdot2N\cdot t-\left(\abs{\eta(a)}_{X(\FF_{\ell})}+\abs{\eta(b)}_{X(\FF_{\ell})}\right)\cdot t}{t}=$

$=2N\cdot tr_{X(\FF_{\ell})}(\eta(z))-\left(\abs{\eta(a)}_{X(\FF_{\ell})}+\abs{\eta(b)}_{X(\FF_{\ell})}\right)$.

Thus, if BWOC $\abs{k_{n}}\underset{n\rightarrow\infty}{\rightarrow}\infty$,
then for large enough $n\in\mathbb{N}$,

\hspace{-0.5cm}$tr_{\FF_{\ell}}(\eta(\pp_{n}(v)))\geq2\abs{k_{n}}\cdot tr_{\FF_{\ell}}(\eta(z))-(\abs{\eta(a)}_{X(\FF_{\ell})}+\abs{\eta(b)}_{X(\FF_{\ell})})$.
Using this and the aforementioned properties of translation lengths,
we see that

\hspace{-0.5cm}$tr_{\FF}(\iota_{n}\circ h_{n}\circ\eta\circ\pp_{n}(v))=tr_{\FF}(h_{n}\circ\eta\circ\pp_{n}(v))\geq(\xi_{n}-2c_{n})\cdot tr_{\FF_{\ell}}(\eta(\pp_{n}(v)))\geq$

$\geq(\xi_{n}-2c_{n})\cdot\left(2\abs{k_{n}}\cdot tr_{\FF_{\ell}}(\eta(z))-(\abs{\eta(a)}_{X(\FF_{\ell})}+\abs{\eta(b)}_{X(\FF_{\ell})})\right)$.

On the other hand, $v$ can be written as a word $v=u_{v_{1}}\cdot...\cdot u_{v_{m}}$
in the letters of the generating set $\boldsymbol{u}$. So $\abs{\iota_{n}\circ h_{n}\circ\eta\circ\pp_{n}(v)}\leq\underset{i=1}{\overset{m}{\sum}}\abs{\iota_{n}\circ h_{n}\circ\eta\circ\pp_{n}(u_{v_{i}})}\leq m\cdot\underset{u\in\boldsymbol{u}}{max}\,\abs{h_{n}\circ\eta(u)}\leq m\cdot\chi_{n}\cdot\underset{u\in\boldsymbol{u}}{max}\,\abs{\eta(u)}$
(the middle inequality due to the fact that $\iota_{n},\,\pp_{n}$
were chosen so as to $\boldsymbol{u}$-shorten $h_{n}\circ\eta$).
Since translation length is upper-bounded by displacement length,
it follows that

\hspace{-0.5cm}$(\xi_{n}-2c_{n})\cdot\left(2\abs{k_{n}}\cdot tr_{\FF_{\ell}}(\eta(z))-(\abs{\eta(a)}_{X(\FF_{\ell})}+\abs{\eta(b)}_{X(\FF_{\ell})})\right)\leq m\cdot\chi_{n}\cdot\underset{u\in\boldsymbol{u}}{max}\,\abs{\eta(u)}$.
By rearranging the inequality and taking limsup on both sides:

\hspace{-0.5cm}$\underset{n\rightarrow\infty}{limsup}\,\abs{k_{n}}\leq\left(m\cdot\underset{u\in\boldsymbol{u}}{max}\,\abs{\eta(u)}_{X(\FF_{\ell})}+(\abs{\eta(a)}_{X(\FF_{\ell})}+\abs{\eta(b)}_{X(\FF_{\ell})})\right)/2tr_{\FF_{\ell}}(\eta(z))$.

\hspace{-0.5cm}(It is possible to divide by $tr_{\FF_{\ell}}(\eta(z))$
because $\eta$ is strict and in particular $\eta\vert_{A}$ is injective,
hence $\eta(z)\neq1_{\FF_{\ell}}$ and has nonzero translation length
in $X(\FF_{\ell})$.) So $\abs{k_{n}}$ is eventually bounded, a contradiction.

This gives a global bound on the set $\{\abs{k_{n}}\}_{\nn}$, and
hence also for $\{\pp_{n}\}_{\nn}$. By extraction of subsequence,
it may be assumed that $\{\pp_{n}\}_{\nn}$ is the constant sequence
$\pp_{n}=\pp_{0}$. Since $\eta\circ\pp_{0}$ is a composition of
an epimorphism and an isomorphism, it is surjective. Therefore:

\hspace{-0.5cm}$\FF_{\ell}\cong\eta\circ\pp_{o}(L)=\eta\circ\pp_{o}(L)/\sk h_{n}=\eta\circ\pp_{o}(L)/\sk(\iota_{n}\circ h_{n})=$

$=L/\sk(\iota_{n}\circ h_{n}\circ\eta\circ\pp_{o})=L/\sk(\iota_{n}\circ h_{n}\circ\eta\circ\pp_{n})$.

\hspace{-0.5cm}So $(\FF_{\ell},\eta\circ\pp_{0})$ is a $\boldsymbol{u}$-shortening
quotient, and it is SQ-isomorphic to $(\FF_{\ell},\eta)$, as seen
by the commutativity of the diagram 

\hspace{-0.5cm}$\begin{array}{ccc}
L & \overset{\pp_{1}}{\longrightarrow} & L\\
{\scriptstyle \eta\circ\pp_{0}}\downarrow & \searrow^{\eta\circ\pp_{1}} & \downarrow{\scriptstyle \eta}\\
\FF_{\ell} & \underset{id}{\longrightarrow} & \FF_{\ell}
\end{array}$.

\subsection{Second case - two levels of single-edged JSJ decompositions}

Suppose $L$ is a limit group whose cyclic JSJ decomposition is $L=A\agm{\langle z\rangle}B$,
where $A,B$ are both rigid and $\langle z\rangle$ is maximal abelian
in both vertex groups. Let $\boldsymbol{g},\boldsymbol{u}$ be two
finite generating sets, and let $(Q,\eta)$ be a strict maximal $\boldsymbol{g}$-shortening
quotient of $L$ with a cyclic JSJ decomposition $Q=D\agm{\langle c\rangle}E$
such that $D,E$ are both rigid and $\langle c\rangle$ is maximal
abelian in both vertex groups. Denote $\tau_{z}$ the Dehn twist along
$z$. Since $z$ generates the only edge group of $L$'s JSJ decomposition
and the vertex groups are rigid, $Mod(L)$ is generated by $\tau_{z}$
and $Inn(L)$. Also assume that the next level of the strict resolution
is $(\FF_{\ell},\pi)$, and fix some free generating set $\boldsymbol{x}=(x_{1},...,x_{\ell})$.

\begin{lem}
$\eta(z)$ is hyperbolic in $Q$.
\end{lem}
\begin{proof}
Otherwise, both $\eta(A)$ and $\eta(B)$ must be elliptic too, due
to the rigidity of $A$ and $B$. Clearly $A$ and $B$ cannot both
embed into the same vertex group of $Q$, as that would contradict
surjectivity. Therefore WLOG $\eta(A)\subseteq D$ and $\eta(B)\subseteq X^{p}$
for $X\in\{D,E\}$ and $p\in D\agm{\langle c\rangle}E$ ($X^{p}$
is the conjugation of $X$ by $p$). Since $\langle c\rangle$ is
maximal abelian in both vertex groups, it follows that the Bass-Serre
tree associated with $D\agm{\langle c\rangle}E$ is 1-acylindrical.
As $1\neq\eta(z)\in\eta(A)\cap\eta(B)\subseteq D\cap X^{p}$, by 1-acylindricity
the vertices stabilised by $D$ and by $X^{p}$ must be neighbours.
Specifically, $X=E$ and WLOG $p=1$.

Now that $\eta(A)\subseteq D$ and $\eta(B)\subseteq E$, it is possible
to show that $\eta$ is injective: let $1\neq v\in A\agm{\langle z\rangle}B$
have the normal form $v=a_{1}b_{1}...a_{\ell}b_{\ell}$, i.e. $a_{i}\in A$
and $b_{i}\in B$ s.t. $[a_{i},z]\neq1$, $[b_{i},z]\neq1$ for $1\leq i\leq\ell$
(except perhaps $a_{1}=1$ or $b_{\ell}=1$, but ignore this case
for simplicity of notation). $\eta$ is strict, and in particular
$\eta\vert_{A}$ and $\eta\vert_{B}$ are injective, so $\eta([a_{i},z])\neq1$
and $\eta([b_{i},z])\neq1$. In addition, $\eta(a_{i})\in\eta(A)\subseteq D$
and $\eta(b_{i})\in\eta(B)\subseteq E$. It follows that $\eta(v)=\eta(a_{1})\eta(b_{1})...\eta(a_{\ell})\eta(b_{\ell})$
is a non-trivial normal form, so $\eta(v)\neq1$. Hence $Ker(\eta)=\{1\}$
and $\eta$ is injective.

The quotient map $\eta$ is therefore an isomorphism. This contradicts
the fact that shortening quotients are proper quotients.

\end{proof}
Hyperbolicity is a property of the entire conjugacy class of an element.
The fact that $\eta(z)$ is hyperbolic means that the shortest normal
form of a cyclically reduced conjugate $\tilde{z}$ is $\tilde{z}=d_{1}e_{1}...d_{m_{z}}e_{m_{z}}$,
where $d_{i}\in D\backslash\langle c\rangle$, $e_{i}\in E\backslash\langle c\rangle$
for all $1\leq i\leq m_{z}$, and $1\leq m_{z}$. (If the normal
form were to begin and end with nontrivial elements both from $D$,
for example, then by conjugation, $(d_{1})^{-1}\tilde{z}d_{1}$ is
a normal form that is shorter than the previous one, a contradiction.)

$ $

Let $\{\aij h1n\}_{\nn}\subseteq Hom(Q,\FF)$ be a test sequence relative
to the JSJ decompositions of the resolution $Q\overset{\pi}{\rightarrow}\FF_{\ell}$,
so there exist $\{\aij h2n\}_{\nn}\subseteq Hom(\FF_{\ell},\FF)$
such that:
\begin{itemize}
\item for any $f\in\FF_{\ell}$, $(\xi_{n}-2c_{n})\cdot\abs f_{X(\FF_{\ell})}\leq\abs{\aij h2n(f)}\leq\chi_{n}\cdot\abs f_{X(\FF_{\ell})}$;
\item there exist $\{\sigma_{n}\}_{\nn}\subseteq Mod(Q)$ such that $\aij h1n=\aij h2n\circ\pi\circ\sigma_{n}$.
In fact $\sigma_{n}=\rho_{c}^{t_{n}}$ where $\rho_{c}$ is the Dehn
twist $\rho_{c}(f)=\begin{cases}
\begin{array}{c}
f\\
cfc^{-1}
\end{array} & \begin{array}{c}
f\in D\\
f\in E
\end{array}\end{cases}$ and $\{t_{n}\}_{\nn}\subseteq\mathbb{Z}$ is a sequence whose absolute
values tend to infinity.
\end{itemize}
It is also worth noting the following observations regarding $\aij h1n$
(derived from property 3 in the definition of test sequences):
\begin{itemize}
\item Since $\sk h_{n}^{(1)}=1$, elements which do not commute in $Q$
have images in $\FF$ which eventually do not commute, and thus do
not share an axis. 
\item Since $\eta(z)$ is hyperbolic in the JSJ decomposition of $Q$, its
has non-trivial translation length in the Bass-Serre tree $T$ associated
with the cyclic JSJ decomposition of $Q$. Therefore, eventually also
$\aij h1n(\eta(z))$ has non-trivial translation length in $X(\FF)$. 
\end{itemize}
$ $

Take $\{\pp_{n}\}_{\nn}\in Mod(L)$ and $\{\iota_{n}\}_{\nn}\subseteq Inn(\FF)$
such that $\{\iota_{n}\circ(h_{n}\circ\pi\circ\sigma_{n}\circ\eta)\circ\pp_{n}\}_{\nn}$
are $\boldsymbol{u}$-shortest. $\tau_{z}$ and $Inn(L)$ generate
$Mod(L)$, but again by Remark \ref{rem:not Inn} it may be assumed
that  there exist $\{k_{n}\}_{\nn}\subseteq\mathbb{Z}$ such that
$\pp_{n}=\tau_{z}^{k_{n}}.$ Again, the aim is to show that the sequence
$\{\pp_{n}\}_{\nn}$ is comprised of a finite set of homomorphisms,
by proving that $\abs{k_{n}}$ is eventually bounded. By way of contradiction,
assume that $\abs{k_{n}}$ tends to infinity.

$ $

First analyse $\pi\circ\sigma_{n}(\tilde{z})$: recall the normal
form $\tilde{z}=d_{1}e_{1}d_{2}e_{2}...d_{m_{z}}e_{m_{z}}$ , which
has more than one nontrivial letter because $\tilde{z}$ is hyperbolic.
 Also note that since $c\in D\backslash\{1_{Q}\}$ and $\pi$ is
strict, it follows that $\pi(c)\neq1_{\FF_{\ell}}$, and therefore
$\pi(c)$ has nonzero translation length in $X(\FF_{\ell})$. Denote\linebreak{}
$J=\underset{i=1}{\overset{m_{z}}{\sum}}(\abs{\pi(d_{i})}_{X(\FF_{\ell})}+\abs{\pi(e_{i})}_{X(\FF_{\ell})})$
.

Because the normal form of $\tilde{z}$ finishes with a letter from
$E$ and starts with a letter from $D$, it follows that the normal
form of $\tilde{z}^{k_{n}}$ is simply $k_{n}$ consecutive copies
of the normal form of $\tilde{z}$. Therefore:

\hspace{-0.5cm}$tr_{\FF_{\ell}}\left(\pi\circ\sigma_{n}\circ\eta(z^{k_{n}})\right)=tr_{\FF_{\ell}}\left(\pi\circ\sigma_{n}(\tilde{z}^{k_{n}})\right)=tr_{\FF_{\ell}}\left(\pi\left((d_{1}c^{t_{n}}e_{1}c^{-t_{n}}...d_{m_{z}}c^{t_{n}}e_{m_{z}}c^{-t_{n}})^{k_{n}}\right)\right)\geq$

$\geq2m_{z}\cdot\abs{k_{n}}\cdot\abs{t_{n}}\cdot tr_{\FF_{\ell}}(\pi(c))-\abs{k_{n}}\cdot\left(\underset{i=1}{\overset{m_{z}}{\sum}}(\abs{\pi(d_{i})}_{X(\FF_{\ell})}+\abs{\pi(e_{i})}_{X(\FF_{\ell})})\right)=$

$=\abs{k_{n}}\cdot(2m_{z}\cdot\abs{t_{n}}\cdot tr_{\FF_{\ell}}(\pi(c))-J)$

$ $

As in the previous case, let $v=ab$ where $a\in A\backslash\langle z\rangle$
and $b\in B\backslash\langle z\rangle$. Both $\eta(a)$ and $\eta(b)$
have normal forms with respect to the JSJ decomposition of $Q$: $\eta(a)=\aij da1\aij ea1\cdot...\cdot\aij da{m_{a}}\aij ea{m_{a}}\aij da{m_{a}+1}$
and\linebreak{}
$\eta(b)=\aij db1\aij eb1\cdot...\cdot\aij db{m_{b}}\aij eb{m_{b}}\aij db{m_{b}+1}$,
where $\aij dai,\aij dbi\in D\backslash\langle c\rangle$ and $\aij eai,\aij ebi\in E\backslash\langle c\rangle$,
except \linebreak{}
perhaps $\aij da1$, $\aij db1$, $\aij da{m_{a}+1}$ and $\aij db{m_{b}+1}$,
each of which may equal $1_{Q}$. Denote\linebreak{}
$Y=\{\aij daj,\aij dbi,\aij eaj,\aij ebi:1\leq j\leq m_{a},1\leq i\leq m_{b}\}\cup\{\aij da{m_{a}+1},\aij db{m_{b}+1}\}$
and $K=\underset{y\in Y}{max}\,\abs{\pi(y)}_{X(\FF_{\ell})}$ .

\hspace{-0.5cm}$\abs{\pi\circ\sigma_{n}\circ\eta(a)}_{X(\FF_{\ell})}\leq\underset{i=1}{\overset{m_{a}}{\sum}}\left(\abs{\pi(\aij dai)}_{X(\FF_{\ell})}+\abs{\pi(\aij eai)}_{X(\FF_{\ell})}\right)+\abs{\pi(\aij da{m_{a}+1})}_{X(\FF_{\ell})}+2m_{a}\cdot\abs{t_{n}}\cdot\abs{\pi(c)}_{X(\FF_{\ell})}\leq$

$\leq(2m_{a}+1)\underset{y\in Y}{max}\,\abs{\pi(y)}_{X(\FF_{\ell})}+2m_{a}\cdot\abs{t_{n}}\cdot\abs{\pi(c)}_{X(\FF_{\ell})}=(2m_{a}+1)K+2m_{a}\cdot\abs{t_{n}}\cdot\abs{\pi(c)}_{X(\FF_{\ell})}$
.

\hspace{-0.5cm}Likewise, $\abs{\pi\circ\sigma_{n}\circ\eta(b)}_{X(\FF_{\ell})}\leq(2m_{b}+1)K+2m_{b}\cdot\abs{t_{n}}\cdot\abs{\pi(c)}_{X(\FF_{\ell})}$
.

Now, recall that any two elements of $Q$ which do not share an axis
in $T$, eventually have images that do not share an axis in $\aij h1n(Q)$.
 $a$ and $b$ do not commute with $z$, so by strictness, $\eta(a)$
and $\eta(b)$ both do not commute with $\eta(z).$ Therefore, neither
$\eta(a)$ nor $\eta(b)$ share an axis with $\eta(z)$ in $T$. Therefore
$tr_{T}(\eta\circ\pp_{n}(v))\geq2\cdot tr_{T}(\eta(z^{k_{n}}))-(\abs{\eta(a)}_{T}+\abs{\eta(b)}_{T})$.
This behaviour is reflected by $\aij h1n(Q)$: $tr_{\FF}(\aij h1n(\eta\circ\pp_{n}(v)))\geq2\cdot tr_{\FF}(\aij h1n(\eta(z^{k_{n}})))-(\abs{\aij h1n(\eta(a))}+\abs{\aij h1n(\eta(b))})$.
Consequently, the following holds:

$ $

\hspace{-0.5cm}$tr_{\FF}\left(\iota_{n}\circ\aij h2n\circ\pi\circ\sigma_{n}\circ\eta\circ\pp_{n}(v)\right)=tr_{\FF}\left(\aij h2n\circ\pi\circ\sigma_{n}\circ\eta\circ\pp_{n}(v)\right)=$

$=tr_{\FF}(\aij h1n(\eta\circ\pp_{n}(v)))\geq2\cdot tr_{\FF}(\aij h1n(\eta(z^{k_{n}})))-(\abs{\aij h1n(\eta(a))}+\abs{\aij h1n(\eta(b))})=$

$=2\cdot tr_{\FF}(\aij h2n\circ\pi\circ\sigma_{n}\circ\eta(z^{k_{n}}))-(\abs{\aij h2n\circ\pi\circ\sigma_{n}\circ\eta(a)}+\abs{\aij h2n\circ\pi\circ\sigma_{n}\circ\eta(b)})\geq$

$\geq2(\xi_{n}-2c_{n})\cdot tr_{\FF_{\ell}}(\pi\circ\sigma_{n}\circ\eta(z^{k_{n}}))-\chi_{n}\cdot(\abs{\pi\circ\sigma_{n}\circ\eta(a)}_{X(\FF_{\ell})}+\abs{\pi\circ\sigma_{n}\circ\eta(b)}_{X(\FF_{\ell})})\geq$

$\geq2(\xi_{n}-2c_{n})\cdot\abs{k_{n}}\cdot\left(2m_{z}\cdot\abs{t_{n}}\cdot tr_{\FF_{\ell}}(\pi(c))-J\right)-\chi_{n}\cdot\left(2K(m_{a}+m_{b}+1)+2(m_{a}+m_{b})\cdot\abs{t_{n}}\cdot\abs{\pi(c)}_{X(\FF_{\ell})}\right)$.

$ $

On the other hand, $v$ can be written as a word $v=u_{v_{1}}\cdot...\cdot u_{v_{m}}$
in the letters of the generating set $\boldsymbol{u}$. For every
$u\in\boldsymbol{u}$ there exists a normal form $\eta(u)=\aij du1\aij eu1...\aij du{m_{u}}\aij eu{m_{u}}$
of $\eta(u)$\linebreak{}
($\aij duj\in D\backslash\langle c\rangle$ and $\aij euj\in E\backslash\langle c\rangle$,
except perhaps $\aij du1$ and $\aij eu{m_{u}}$, each of which may
equal $1_{Q}$). Denote $R=\left\{ \aij duj,\aij euj:\,1\leq j\leq m_{u},\,u\in\boldsymbol{u}\right\} $.
 So

\hspace{-0.5cm}$\abs{\iota_{n}\circ\aij h2n\circ\pi\circ\sigma_{n}\circ\eta\circ\pp_{n}(v)}\leq\underset{i=1}{\overset{m}{\sum}}\abs{\iota_{n}\circ\aij h2n\circ\pi\circ\sigma_{n}\circ\eta\circ\pp_{n}(u_{v_{i}})}\leq m\cdot\underset{u\in\boldsymbol{u}}{max}\,\abs{\aij h2n\circ\pi\circ\sigma_{n}\circ\eta(u)}\leq$

$\leq m\cdot\chi_{n}\cdot\underset{u\in\boldsymbol{u}}{max}\,\abs{\pi\circ\sigma_{n}\circ\eta(u)}_{X(\FF_{\ell})}\leq m\cdot\chi_{n}\cdot(\underset{u\in\boldsymbol{u}}{max}\,\{m_{u}\})\cdot\left(\underset{r\in R}{max}\,\abs{\pi(r)}_{X(\FF_{\ell})}+2\abs{t_{n}}\cdot\abs{\pi(c)}_{X(\FF_{\ell})}\right)$
.

\hspace{-0.5cm}Since the translation length is upper-bounded by the
displacement length, it follows that\linebreak{}
\hspace{-0.5cm}$2(\xi_{n}-2c_{n})\cdot\abs{k_{n}}\cdot\left(2m_{z}\cdot\abs{t_{n}}\cdot tr_{\FF_{\ell}}(\pi(c))-J\right)-\chi_{n}\cdot\left(2K(m_{a}+m_{b}+1)+2(m_{a}+m_{b})\cdot\abs{t_{n}}\cdot\abs{\pi(c)}_{X(\FF_{\ell})}\right)\leq$

$\leq m\cdot\chi_{n}\cdot(\underset{u\in\boldsymbol{u}}{max}\,\{m_{u}\})\cdot\left(\underset{r\in R}{max}\,\abs{\pi(r)}_{X(\FF_{\ell})}+2\abs{t_{n}}\cdot\abs{\pi(c)}_{X(\FF_{\ell})}\right)$
.

\hspace{-0.5cm}By rearranging the inequality and taking limsup on
both sides:

\hspace{-0.5cm}$\underset{n\rightarrow\infty}{limsup}\,\abs{k_{n}}\leq\left((m_{a}+m_{b}+m\cdot(\underset{u\in\boldsymbol{u}}{max}\,\{m_{u}\}))\cdot\abs{\pi(c)}_{X(\FF_{\ell})}\right)/\left(2m_{z}\cdot tr_{\FF_{\ell}}(\pi(c))\right)$
.

$ $

As before, by extraction assume that $\{\pp_{n}\}_{\nn}$ is some
constant sequence $\pp_{n}=\pp_{0}$. 

\hspace{-0.5cm}$Q\cong\eta\circ\pp_{o}(L)=\eta\circ\pp_{o}(L)/\sk h_{n}^{(1)}=\eta\circ\pp_{o}(L)/\sk(\iota_{n}\circ h_{n}^{(1)})=$

$=L/\sk(\iota_{n}\circ h_{n}^{(1)}\circ\eta\circ\pp_{o})=L/\sk(\iota_{n}\circ h_{n}^{(1)}\circ\eta\circ\pp_{n})$

\hspace{-0.5cm}So $(Q,\eta)$ is SQ-isomorphic to $(Q,\eta\circ\pp_{0})$.
\begin{rem}
In fact, this analysis, which has been illustrated for a resolution
of length 2, can be extended to an analysis of a similarly built resolution
of length $\nn$. If every group along the resolution has a single-edged
JSJ decomposition with rigid vertex groups, and the edge group is
maximal abelian in the vertex groups, then each edge generator is
hyperbolic with respect to the next group along the resolution. Iteratively,
at every level the dominant factor is the translation length of the
generator of the edge group of that level, multiplied by the number
of appearances which are made by each edge group, from that level
upwards, in the normal form of the cyclically reduced conjugate of
the edge group generator of the preceding level.
\end{rem}

\end{document}